\documentclass[a4paper]{amsart}
\usepackage[all]{xy}

\newtheorem{theorem}{Theorem}
\newtheorem{proposition}{Proposition}
\newtheorem{corollary}{Corollary}
\newtheorem{lemma}{Lemma}

\theoremstyle{remark}
\newtheorem{remark}{Remark}
\newtheorem{example}{Example}
\newtheorem*{acknowledgement}{Acknowledgements}

\newcommand{\an}{\mathcal{O}}

\newcommand{\p}{\mathbb {P}}
\newcommand{\g}{\mathbb {G}}

\newcommand{\length}{\operatorname{length}}
\newcommand{\Supp}{\operatorname{Supp}}
\newcommand{\charac}{\operatorname{char}}

\begin{document}

\title{A degree bound for globally generated vector bundles}

\author{Jos\'e Carlos Sierra}

\address{\hskip -.43cm Jos\'e Carlos Sierra, Departamento de \'Algebra, Facultad de Ciencias Matem\'aticas, Universidad
Complutense de Madrid, 28040 Madrid, Spain. e-mail {\tt
jcsierra@mat.ucm.es}}

\thanks{Research partially supported by MCYT project MTM2006-04785 and by the program ``Profesores de la UCM en el extranjero. Convocatoria 2006".}

\begin{abstract}
Let $E$ be a globally generated vector bundle of rank $e\geq 2$
over a reduced irreducible projective variety $X$ of dimension $n$ defined over an
algebraically closed field of characteristic zero. Let
$L:=\det(E)$. If $\deg(E):=\deg(L)=L^n>0$ and $E$ is not
isomorphic to $\an_X^{e-1}\oplus L$, we obtain a sharp bound
$$\deg(E)\geq h^0(X,E)-e$$ on the degree of $E$, proving also that
$\deg(L)=h^0(X,L)-n$ if equality holds. As an application, we
obtain a Del Pezzo-Bertini type theorem on varieties of minimal
degree for subvarieties of Grassmannians, as well as a bound on
the sectional genus for subvarieties $X\subset\g(k,N)$ of degree
at most $N+1$. 
\end{abstract}

\maketitle

\section{Introduction}
\label{intro} Let $X$ be a reduced irreducible projective variety
of dimension $n$ defined over an algebraically closed field $k$ of
characteristic zero.

If $X\subset\p^N$ is a non-degenerate embedding (i.e. not contained
in any hyperplane) then $$\deg(X)\geq N+1-n,$$ and $X\subset\p^N$ is
said to be a variety of minimal degree if equality holds. These
varieties were classified by Del Pezzo \cite{dp} for $n=2$, and
Bertini \cite{b} for $n\geq 3$ (see also \cite{eh} for a modern
account). A variety of minimal degree turns out to be either a
linear space, or a quadric hypersurface, or a rational normal
scroll, or (a cone over) the Veronese surface in $\p^5$.

More generally, if $L$ is a globally generated line bundle on $X$
and $\deg(L):=L^n > 0$, then $$\deg(L)\geq h^0(X,L)-n,$$ and, if
equality holds, then $\varphi_{|L|}:X\to
X'\subset\p^{h^0(X,L)-1}$ is a birational morphism onto a variety
of minimal degree, so in particular $X$ is rational.

In this note, a similar bound is obtained for a globally generated
vector bundle $E$ of rank $e\geq 2$ over $X$ such that
$\deg(E):=\deg(L)=L^n>0$, where $L:=\det(E)$. If
$E\cong\an_{X}^{e-1}\oplus L$ then
\[
\deg(E)=\deg(L)\geq h^0(X,L)-n=h^0(X,E)-e+1-n.
\]
Therefore, we assume $E$ different from $\an_{X}^{e-1}\oplus L$ in
the sequel. The main result of the paper is the following:

\begin{theorem}\label{thm:main}
Let $E$ be a globally generated vector bundle of rank $e\geq 2$ on
$X$ and let $L:=\det(E)$. If $\deg(E)>0$ and $E$ is not isomorphic
to $\an_X^{e-1}\oplus L$, then
\[
\deg(E)\geq h^0(X,E)-e.
\]
Moreover, if equality holds then $\deg(L)=h^0(X,L)-n$ and, in
particular, $X$ is rational.
\end{theorem}

We remark that this bound does not depend on the dimension of $X$,
as one might expect a priori.

\smallskip

For general Cartier divisors $Y_1,\dots,Y_{n-1}\in |L|$, let
$C:=\cap_{i=1}^{n-1} Y_i$ denote the curve section (consider $C=X$
for $n=1$) and let $p_a(C)$ denote its arithmetic genus. As a
byproduct of the proof of Theorem \ref{thm:main}, we also obtain
the following result:

\begin{corollary}\label{cor:main}
Let $E$ be a globally generated vector bundle of rank $e\geq 2$ on
$X$ and let $E^*$ denote its dual bundle. Assume $h^0(X,E)>e+1$
for $n\geq 2$. If $\deg(E)>0$ and $E$ is not isomorphic to
$\an_X^{e-1}\oplus L$, then:
\begin{enumerate}
\item[(i)] if $\deg(E)<h^0(X,E)-h^0(X,E^*)$ then $C\cong\p^1$. Moreover, if $X$ is normal then $\deg(L)=h^0(X,L)-n$ and $X$ is rational;
\item[(ii)] if $\deg(E)=h^0(X,E)-h^0(X,E^*)$ then either $p_a(C)\leq 1$, or $C$ is
hyperelliptic and the restriction of $E$ to $C$ satisfies
$E_C\cong\an_C^r\oplus g^1_2\oplus\dots\oplus g^1_2$, where
$r:=h^0(C,E_C^*)$ and $g^1_2$ denotes the pencil of degree $2$ on
$C$.
\end{enumerate}
\end{corollary}

\begin{remark}\label{rem:n=1}
For $n=1$, both Theorem \ref{thm:main} and Corollary \ref{cor:main}
hold under weaker hypotheses (see Lemma \ref{lem:clifford} and
Proposition \ref{prop:hyperelliptic}).
\end{remark}

As an application of Theorem \ref{thm:main}, we extend Del
Pezzo-Bertini's theorem to non-degenerate subvarieties
$X\subset\g(k,N)$ of Grassmanians of $k$-planes in $\p^N$. In this
context, \emph{non-degenerate} means that $X$ is not contained in
any $\g(k,N-1)$:

\begin{corollary}\label{cor:grass}
Let $X\subset\g(k,N)$ be a non-degenerate embedding. If
$X\subset\g(k,N)$ is not contained in any Schubert variety $\Omega
(k-1,N)$ of $k$-planes containing a linear subspace
$\p^{k-1}\subset\p^N$ then
\[
\deg(X)\geq N-k,
\]
where $\deg (X)$ is the degree of $X\subset\p^M$ given by the
Pl\"ucker embedding of $\g(k,N)$. Moreover, if equality holds then
$X\subset\p^M$ is a variety of minimal degree in its linear span.
\end{corollary}

Furthermore, we can also rephrase Corollary \ref{cor:main} for
subvarieties of Grassmannians. Let $X\subset\g(k,N)$, and let
$\an_X(1)$ denote the line bundle on $X$ giving the Pl\"ucker
embedding $X\subset\p^M$. For general Cartier divisors
$H_1,\dots,H_{n-1}\in \an_X(1)$, let $C:=\cap_1^{n-1}H_i$ denote
the curve section (consider $C=X$ for $n=1$) and let $p_a(C)$
denote its arithmetic genus:

\begin{corollary} \label{cor:N+1}
Let $X\subset\g(k,N)$ be a non-degenerate embedding. Assume $N-k\geq
2$ for $n\geq 2$. Let $s\leq k-1$ be a non-negative integer. If
$X\subset\g(k,N)$ is not contained in any Schubert variety
$\Omega(s,N)$ of $k$-planes containing a linear subspace
$\p^s\subset\p^N$, then:
\begin{enumerate}
\item[(i)] if $\deg(X)\leq N-s$ then $C\cong\p^1$. Moreover, if $X$
is normal then $X\subset\p^M$ is (a linear projection of) a variety
of minimal degree in its linear span;
\item[(ii)] if $\deg(X)=N+1-s$ then $p_a(C)\leq1$.
\end{enumerate}
\end{corollary}

It is proved in \cite{ion} that non-degenerate manifolds
$X\subset\p^N$ of degree $\deg(X)\leq N$ are simply connected.
This topological result follows as a corollary of a complete
classification of such manifolds. According to Corollary
\ref{cor:N+1}, the Pl\"ucker embedding of a non-degenerate normal
subvariety $X\subset\g(k,N)$ of degree $\deg(X)\leq N$ not
contained in any Schubert variety $\Omega(0,N)$ is either a linear
space, or a quadric hypersurface, or a rational scroll, or (a cone
over) the Veronese surface. Hence, we also deduce \emph{a
posteriori} a similar result for subvarieties of Grassmannians:

\begin{corollary}\label{cor:simply}
Assume $N-k\geq 2$ for $n\geq 2$. Then, a non-degenerate normal
subvariety $X\subset\g(k,N)$ of degree $\deg(X)\leq N$ not
contained in any Schubert variety $\Omega(0,N)$ is simply
connected.
\end{corollary}

The bound $\deg(X)\leq N$ is optimal in both \cite{ion} and
Corollary \ref{cor:simply}. Indeed, for any $r\geq 1$ there exist
$r$-dimensional elliptic scrolls in $\p^{2r}$ of degree $2r+1$.
Considering the corresponding elliptic curve in $\g(r-1,2r)$ of
degree $2r+1$, we obtain that also our bound is sharp. On the other
hand, if we consider a cone in $\p^{2r+1}$ of vertex a point over an
$r$-dimensional elliptic scroll in $\p^{2r}$ of degree $2r+1$, the
corresponding elliptic curve of degree $2r+1$ in $\g(r,2r+1)$ shows
that the assumption \emph{$X$ not contained in any Schubert variety
$\Omega(0,N)$} cannot be dropped.

\smallskip

The paper is organized as follows. Section \ref{section:proof} is
devoted to the proof of Theorem \ref{thm:main} and Corollary
\ref{cor:main}. To this purpose, the crucial point here is to
prove that $h^0(X,E)\leq h^0(C,E_C)$ whenever $E$ is not
isomorphic to $\an_X^{e-1}\oplus L$ and $\deg(E)>0$ (see Corollary
\ref{cor:curve}). Then we obtain a degree bound for globally
generated vector bundles on curves (see Lemma \ref{lem:clifford}
and Proposition \ref{prop:hyperelliptic}). Using the
correspondence between globally generated vector bundles and maps
to Grassmannians, in Section \ref{section:grass} we translate
Theorem \ref{thm:main} and Corollary \ref{cor:main} into Corollary
\ref{cor:grass} and Corollary \ref{cor:N+1}, respectively. Some
examples of subvarieties on the boundary are also included.

\section{Proof of Theorem \ref{thm:main}}\label{section:proof}
Let $E$ be a globally generated vector bundle of rank $e\geq 2$ on
$X$ and let $L:=\det(E)$.

\begin{proposition}\label{prop:key}
If $h^0(X,E\otimes L^{-1})\neq 0$ then $E\cong\an_X^{e-1}\oplus L$.
\end{proposition}

\begin{proof}
Let $s\in H^0(X,E\otimes L^{-1})$ be a non-zero section. Consider
the corresponding exact sequence of sheaves on $X$
\[
\xymatrix {0\ar[r] & L \ar[r] & E \ar[r] & F \ar[r]&0.}
\]
Let $Z=\{x\in X\mid s(x)=0\}$ denote the zero-locus of $s$. We
claim that $Z=\emptyset$. To get a contradiction, let $x\in Z$ and
let $C\subset X$ be a curve not contained in $Z$ such that
$x\in C$. Restricting the above exact sequence to the normalization $\widetilde C$ of $C$, we obtain
an exact sequence of sheaves
\[
\xymatrix {0\ar[r] & L_{\widetilde C} \ar[r] & E_{\widetilde C} \ar[r] & Q\oplus t
\ar[r]&0},
\]
where $Q$ is a vector bundle of rank $e-1$ on $\widetilde C$ and $t$ is a
torsion sheaf on $\tilde C$. Then
\[
\deg(Q)+\length(t)=\deg(E_{\widetilde C})-\deg(L_{\widetilde C})=0.
\]
This yields a contradiction since $Q$ is globally generated, so
$\deg(Q)\geq 0$, and  $x\in\Supp(t)$, so $\length(t)\geq 1$.

We deduce from $Z=\emptyset$ that $F$ is a globally generated vector
bundle of rank $e-1$ over $X$. Thus there exists a morphism of
vector bundles $u:\an_X^{e-1}\to F$
of rank $e-1$ at some point $x\in X$, and taking determinants we get
a morphism $\det(u):\an_X\to\det(F)$.
Then $u$ has rank $e-1$ at every $x\in X$, since $\det(F)\cong c_1(F)\cong c_1(E)-c_1(L)\cong \an_X$.
Therefore $F\cong\an_X^{e-1}$, whence $E\cong\an_X^{e-1}\oplus L$
since $E$ is globally generated (cf. Remark \ref{rem:^*}). 
\end{proof}

\begin{remark}\label{rem:^*}
Let $E^*$ denote the dual bundle of $E$. If $E$ is a globally
generated vector bundle on $X$ we recall that $h^0(X,E^*)=r$ if and
only if $E\cong\an_X^r\oplus F$, where $F$ is a globally generated
vector bundle with $h^0(X,F^*)=0$. Moreover, if $X$ is a curve, then
a globally generated vector bundle $E$ on $X$ is ample if and only
if $h^0(X,E^*)=0$.
\end{remark}

For any Cartier divisor $Y\in|L|$, let $E_Y$ denote the restriction
of $E$ to $Y$. It follows from Remark \ref{rem:^*} that
$h^0(Y,E^*_Y)\geq h^0(X,E^*)$. Moreover:

\begin{lemma}\label{lem:e+1}
If $h^0(Y,E^*_Y)>h^0(X,E^*)$ for some $Y\in|L|$, then
$h^0(X,E)=e+1$.
\end{lemma}

\begin{proof}
Let $s:=h^0(X,E)$. Consider the exact sequence of vector bundles
\[
\xymatrix {0\ar[r] & G \ar[r] & \an_X^s \ar[r] & E \ar[r]&0,}
\]
where $G$ arises as the kernel of the evaluation morphism. Dualizing
the exact sequence above and restricting to $Y$, we obtain the
following commutative diagram:
\[
\xymatrix{
         & 0                        \ar[d] & 0                          \ar[d] & 0                        \ar[d] &  \\
0 \ar[r] & E^*\otimes L^{-1} \ar[r] \ar[d] & \an_X^s\otimes L^{-1}\ar[r]\ar[d] & G^*\otimes L^{-1} \ar[r] \ar[d] & 0\\
0 \ar[r] & E^*               \ar[r] \ar[d] & \an_X^s              \ar[r]\ar[d] & G^*               \ar[r] \ar[d] & 0\\
0 \ar[r] & E^*_Y             \ar[r] \ar[d] & \an_Y^s              \ar[r]\ar[d] & G^*_Y             \ar[r] \ar[d] & 0\\
         & 0                               & 0                                 & 0                               &  \\
}
\]
If $h^0(Y,E^*_Y)>h^0(X,E^*)$, then a simple diagram chase shows that
$h^0(X,G^*\otimes L^{-1})\neq 0$. It follows from Proposition
\ref{prop:key} that $G^*\cong\an_X^{s-e-1}\oplus L$, where $L=\det(E)\cong\det(G^*)$. In particular $G\cong\an_X^{s-e-1}\oplus L^{-1}$,
whence $s=e+1$ since $h^0(X,G)=0$. 
\end{proof}

Assume furthermore $n\geq 2$ and $\deg(E)>0$. For general Cartier divisors
$Y_1,\dots,Y_{n-1}\in |L|$, let $C:=\cap_{i=1}^{n-1} Y_i$ denote
the curve section. Proposition \ref{prop:key} and Lemma
\ref{lem:e+1} yield the following:

\begin{corollary}\label{cor:curve}
Assume $n\geq 2$ and $\deg(E)>0$. If $E$ is not isomorphic to
$\an_X^{e-1}\oplus L$, then:
\begin{enumerate}
\item[(i)] $h^0(X,E)\leq h^0(C,E_C)$;
\item[(ii)] if $h^0(X,E)>e+1$, then $h^0(X,E^*)=h^0(C,E^*_C)$.
\end{enumerate}
\end{corollary}

\begin{proof}
We first remark that $\cap_{i=1}^j Y_i$ is reduced and irreducible
for every $j\in\{1,\dots,n-1\}$. This follows from  Bertini's
theorems applied to the morphism
$\varphi_{|L|}:X\to\p^{h^0(X,L)-1}$ (see, for instance,
\cite[Corollary 6.11]{jou}). Here the hypothesis $\charac(k)=0$ is
used. If $h^0(X,E)=e+1$ then $h^0(C,E_C)\geq e+1$, since $E_C$ is
globally generated and non-trivial. Assume now $h^0(X,E)>e+1$.
Consider the exact sequence
\[
\xymatrix {0\ar[r] & E\otimes L^{-1} \ar[r] & E \ar[r] & E_{Y_1}
\ar[r]&0.}
\]
It follows from Proposition \ref{prop:key} that $h^0(X,E\otimes
L^{-1})=0$. Hence $h^0(X,E)\leq h^0(Y_1,E_{Y_1})$. On the other
side, Lemma \ref{lem:e+1} yields $h^0(X,E^*)=h^0(Y_1,E^*_{Y_1})$. In
particular, $E_{Y_1}$ cannot be isomorphic to $\an_{Y_1}^{e-1}\oplus
L_{Y_1}$. Combining Proposition \ref{prop:key} and Lemma
\ref{lem:e+1} we recursively obtain
$h^0(X,E)\leq h^0(Y_1,E_{Y_1})\leq h^0(Y_1\cap Y_2,E_{Y_1\cap
Y_2})\leq\dots\leq h^0(C,E_{C})$, and also
$h^0(X,E^*)=h^0(Y_1,E^*_{Y_1})=h^0(Y_1\cap Y_2,E^*_{Y_1\cap
Y_2})=\dots=h^0(C,E^*_{C})$. 
\end{proof}

We now consider globally generated vector bundles on curves. The
following lemma generalizes \cite[Proposition 2]{ion}:

\begin{lemma}\label{lem:clifford}
Let $C$ be a reduced irreducible curve of arithmetic genus
$p_a(C)\geq 1$. Let $E$ be a globally generated vector bundle of
rank $e\geq 1$ over $C$, and let $E^*$ denote its dual bundle.
Then $$\deg(E)\geq h^0(C,E)-h^0(C,E^*).$$
\end{lemma}

\begin{proof}
By Remark \ref{rem:^*}, we may assume that $E$ is ample, and, in
that case, it is enough to prove that $\deg(E)\geq h^0(C,E)$. We
proceed by induction, as in \cite{ion}. Assume $e=1$. If
$h^1(C,E)=0$ the result follows from Riemann-Roch's theorem for
singular curves (see, for instance, \cite[Example 18.3.4~(a)]{f}).
If $h^1(C,E)\neq 0$ then $\deg(E)\geq 2(h^0(C,E)-1)$ by a
generalization of Clifford's theorem (see \cite[Theorem A]{eks}),
so $\deg(E)\geq h^0(C,E)$ since $h^0(C,E)\geq 2$ unless
$E\cong\an_C$. Suppose now $e\geq 2$. If $h^0(C,E)\leq e$ then
$E\cong\an_C^e$ since $E$ is globally generated, so we can assume
$h^0(C,E)\geq e+1$. So, if $p\in C$ is a general point, a general
section $s\in H^0(C,E(-p))$ does not vanish at any singular point
of $C$. Thus, it induces an exact sequence
\[
\xymatrix {0\ar[r] & M \ar[r] & E \ar[r] & E' \ar[r]&0,}
\]
where $M$ is a line bundle of positive degree and $E'$ is an ample
vector bundle of rank $e-1$ over $C$. Thus $\deg(E')\geq
h^0(C,E')$ by the induction hypothesis, and using again
Riemann-Roch's theorem as well as \cite[Theorem A]{eks} it follows
that $\deg(M)\geq h^0(C,M)$. These two facts together yield
$$\deg(E)=\deg(E')+\deg(M)\geq h^0(C,E')+h^0(C,M),$$ and
then $\deg(E)\geq h^0(C,E)$, since $h^0(C,E')\geq
h^0(C,E)-h^0(C,M)$. 
\end{proof}

Furthermore, we state separately the following result that may also
be useful for some other purpose:

\begin{proposition}\label{prop:hyperelliptic}
If equality holds in Lemma \ref{lem:clifford}, then either
$E\cong\an_C^e$, or $p_a(C)=1$, or $C$ is a hyperelliptic curve
and $E\cong\an_C^r\oplus g^1_2\oplus\dots\oplus g^1_2$, where
$r:=h^0(C,E^*)$ and $g^1_2$ denotes the pencil of degree $2$ on
$C$.
\end{proposition}

\begin{proof} According to Remark \ref{rem:^*}, we can assume
$h^0(C,E^*)=0$. We proceed by induction on $e$. Assume $e=1$. If
$h^1(C,E)=0$ then $p_a(C)=1$ by Riemann-Roch's theorem. Suppose
now $h^1(C,E)\neq 0$. Then $h^0(C,E)=\deg(E)\geq 2(h^0(C,E)-1)$ by
\cite[Theorem A]{eks}, whence $h^0(C,E)=2$ and equality holds.
Therefore $C$ is hyperelliptic and $E\cong g^1_2$. If $e\geq 2$
and $\deg(E)=h^0(C,E)$, repeating the argument of Lemma
\ref{lem:clifford}, it follows that necessarily $\deg(M)=h^0(C,M)$
and $\deg(E')=h^0(C,E')$. The induction hypothesis yields that
either $p_a(C)=1$, or $C$ is a hyperelliptic curve and $E'\cong
g^1_2\oplus\dots\oplus g^1_2$. Moreover, in the latter case
\cite[Theorem A]{eks} implies $M\cong g^1_2$. Therefore $E$ is an
extension of $g^1_2\oplus\dots\oplus g^1_2$ by $g^1_2$, and
$h^0(C,E)=2e$. Let us show that $E\cong g^1_2\oplus\dots\oplus
g^1_2$. To the contrary, assume that $E$ corresponds to a non-trivial
extension $\xi\in Ext^1(g^1_2\oplus\dots\oplus g^1_2,g^1_2)\cong
Ext^1(\an_C^{e-1},\an_C)\cong H^1(C,\an_C)^{e-1}.$ Then, there
exists also a non-trivial extension $\xi'\in Ext^1(g^1_2,g^1_2)\cong
H^1(C,\an_C)$, yielding a commutative diagram:
\[
\xymatrix{
         &                                 & 0                               \ar[d] & 0                        \ar[d] &  \\
         &                                 & \oplus ^{e-2} g^1_2  \ar@{=}[r] \ar[d] & \oplus ^{e-2} g^1_2      \ar[d] &  \\
0 \ar[r] & g^1_2         \ar[r] \ar@{=}[d] & E                        \ar[r] \ar[d] & \oplus ^{e-1} g^1_2\ar[r]\ar[d] & 0\\
0 \ar[r] & g^1_2             \ar[r] \ar[d] & F                        \ar[r] \ar[d] & g^1_2              \ar[r]\ar[d] & 0\\
         & 0                               & 0                                      & 0                               &  \\
}
\]
To get a contradiction, we prove that $h^0(C,E)<2e$. To this aim it
is enough to check that $h^0(C,F)<4$. Let $D\in g^1_2$ and consider
the corresponding commutative diagram:
\[
\xymatrix{
         & 0                        \ar[d] & 0                      \ar[d] & 0                        \ar[d] &  \\
0 \ar[r] & \an_C             \ar[r] \ar[d] & F(-D)           \ar[r] \ar[d] & \an_C             \ar[r] \ar[d] & 0\\
0 \ar[r] & g^1_2             \ar[r] \ar[d] & F               \ar[r] \ar[d] & g^1_2             \ar[r] \ar[d] & 0\\
0 \ar[r] & g^1_2\otimes\an_D \ar[r] \ar[d] & F\otimes\an_D   \ar[r] \ar[d] & g^1_2\otimes\an_D \ar[r] \ar[d] & 0\\
         & 0                               & 0                             & 0                               &  \\
}
\]
Looking at the long exact sequences of the first two rows we obtain
the following square $\alpha,\beta,\gamma,\delta$:
\[
\xymatrix{
0 \ar[r] & H^0(C,\an_C) \ar[r] \ar[d] & H^0(C,F(-D)) \ar[r] \ar[d] & H^0(C,\an_C) \ar[r]^{\alpha} \ar[d]^{\gamma} & H^1(C,\an_C)\ar[d]^{\delta} \\
0 \ar[r] & H^0(C,g^1_2)        \ar[r] & H^0(C,F)         \ar[r]        & H^0(C,g^1_2) \ar[r]^{\beta}                  & H^1(C,g^1_2)                \\
}
\]
Then $\alpha(1)=\xi'\neq 0$ and $\delta(\xi')\neq 0$ for a suitable
choice of $D\in g^1_2$, whence $\beta$ is not identically zero.
Therefore $h^0(C,F)<4$ and we get a contradiction. 
\end{proof}

We now obtain our main result:

\begin{proof}[Proof of Theorem \ref{thm:main}]
Using the notation in Corollary \ref{cor:curve}, assume first
that $p_a(C)\geq 1$. From Lemma \ref{lem:clifford}, Remark
\ref{rem:^*} and Corollary \ref{cor:curve} we get:
$$\deg(E)=\deg(E_C)\geq h^0(C,E_C)-h^0(C,E^*_C)\geq h^0(C,E_C)-e\geq h^0(X,E)-e.$$ Moreover, $\deg(E) > h^0(X,E)-e$, since $h^0(C,E^*_C)=e$ implies that
$E_C$ is trivial, so $\deg(E)=0$. Assume now $p_a(C)=0$, so $C\cong\p^1$. Corollary \ref{cor:curve}
and Riemann-Roch's theorem give: $$h^0(X,E)\leq
h^0(C,E_C)=\deg(E_C)+e=\deg(E)+e,$$ proving the first part.

We have seen that if $\deg(E)=h^0(X,E)-e$, then necessarily
$C\cong\p^1$ and $h^0(X,E)=h^0(C,E_C)$. Consider the morphism
$\varphi:X\to\g(k,N)$ given by $E$, where $k=e-1$ and
$N=h^0(X,E)-1$ (cf. Section \ref{section:grass}). Let
$\g(k,N)\subset\p^M$ denote the Pl\"ucker embedding. As $E_C$ is
spanned, there is an exact sequence:
\[
\xymatrix {0\ar[r] & \an^{e-1}_C \ar[r] & E_C \ar[r] &
L_C\ar[r]&0.}
\]
Taking cohomology we deduce that the canonical map $\bigwedge^e
H^0(C,E_C)\to H^0(C,L_C)$ is surjective. Moreover, as we have
seen, the restriction map $H^0(X,E)\to H^0(C,E_C)$ is an
isomorphism. It follows that the composition of the restriction
$\varphi_{|C}$ and the Pl\"ucker embedding is given by the
complete linear system $|L_C|$. So, $\varphi(C)\subset\p^M$ is a
rational normal curve in its linear span. We have
$\deg(L_C)=\deg(E_C)=h^0(C,E_C)-e=N-k$, whence the linear span of
$\varphi(C)\subset\p^M$ is a $\p^{N-k}$ and therefore the linear
span of $\varphi(X)\subset\p^M$ is a $\p^{N-k+n-1}$. Thus
$h^0(X,L)\geq N-k+n=\deg(E)+n=\deg(L)+n$. So we get
$0\geq\deg(L)+n-h^0(X,L)\geq 0,$ and therefore
$\deg(L)=h^0(X,L)-n$. In this case, the corresponding map
$\varphi_{|L|}:X\to X'\subset\p^{h^0(X,L)-1}$ is a birational
morphism onto a variety of minimal degree, so $X$ is rational.
\end{proof}


We proceed with the proof of Corollary \ref{cor:main}:

\begin{proof}[Proof of Corollary \ref{cor:main}]
Corollary \ref{cor:curve} yields $h^0(X,E^*)=h^0(C,E^*_C)$, since
we assume $h^0(X,E)>e+1$ whenever $n\geq 2$. Let us prove (i). It
follows from Corollary \ref{cor:curve} that
$$\deg(E_C)=\deg(E)<h^0(X,E)-h^0(X,E^*)\leq
h^0(C,E_C)-h^0(C,E^*_C).$$ Then $p_a(C)=0$ by Lemma
\ref{lem:clifford}, and so $C\cong\p^1$. If, moreover, $X$ is
normal then it easily follows from \cite[Theorem 2]{mum} that
$\deg(L)=h^0(X,L)-n$. Hence $X$ is rational as before. We now
prove (ii). If $\deg(E)=h^0(X,E)-h^0(X,E^*)$ then $\deg(E_C)\leq
h^0(C,E_C)-h^0(C,E^*_C)$, and the thesis follows from Lemma
\ref{lem:clifford} and Proposition \ref{prop:hyperelliptic}. 
\end{proof}

We remark that the assumption $h^0(X,E)>e+1$ in Corollary
\ref{cor:main} cannot be dropped (cf. Remark \ref{rem:p^k+1}).

\section{A degree bound for subvarieties of
Grassmannians}\label{section:grass}

We first recall the well-known correspondence between globally
generated vector bundles and maps to Grassmannians (see, for
instance, \cite{gh} p. 207).

If $E$ is a globally generated vector bundle of rank $e=k+1$ over
an algebraic variety $X$ and $V\subset H^0(X,E)$ is a linear
subspace of dimension $N+1$, any epimorphism $V\otimes\an_X\to
E\to 0$ determines a morphism $\varphi:X\to\g(k,N)$ such that
$\varphi(X)$ is not contained in any $\g(k,N-1)$. Conversely,
every morphism $\varphi:X\to\g(k,N)$ such that $\varphi(X)$ is not
contained in any $\g(k,N-1)$ corresponds to a globally generated
vector bundle $E$ of rank $k+1$ on $X$ and an epimorphism
$V\otimes\an_X\to E\to 0$. In both cases $E\cong\varphi^*(Q_W)$,
where $Q$ denotes the universal quotient bundle of rank $k+1$ on
$\g(k,N)$ and $Q_W$ denotes its restriction to
$W:=\varphi(X)\subset\g(k,N)$. We simply denote $Q_X$ if $\varphi$
is the inclusion. In particular, for $e=1$ we recover the
correspondence between globally generated line bundles and maps to
projective spaces. Moreover, if $\g(k,N)\subset\p^M$ denotes the
Pl\"ucker embedding of the Grassmannian then its restriction
$W\subset\p^M$ is given by a base point free linear system
contained in $|\an_W(1)|$, where $\an_W(1):=\det(Q_W)$.

\begin{remark}\label{rem:omega}
Let $\Omega(s,N)\subset\g(k,N)$ denote the Schubert variety of
$k$-planes containing a linear subspace $\p^s\subset\p^N$. Then,
$X\subset\g(k,N)$ is contained in a $\Omega(s,N)$ if and only if
$Q_X$ has a trivial quotient $\an_X^{s+1}$, and this happens if and
only if $E\cong\an_X^{s+1}\oplus F$, where $F$ is a globally
generated vector bundle on $X$ of rank $k-s$ (cf. Remark
\ref{rem:^*}).
\end{remark}

Theorem \ref{thm:main} allows us to extend Del Pezzo-Bertini's
theorem on varieties of minimal degree to subvarieties of
Grassmannians:

\begin{proof}[Proof of Corollary \ref{cor:grass}]
We deduce from Remark \ref{rem:omega} and Theorem \ref{thm:main}
that $$\deg(X)=\deg(Q_X)\geq h^0(X,Q_X)-e,$$ so $\deg(X)\geq N-k$
since $h^0(X,Q_X)\geq N+1$ and $e=k+1$. Furthermore, if
$\deg(X)=N-k$ the linear span of $X\subset\p^M$ is a $\p^{N-k+n-1}$,
as we concluded at the end of the proof of Theorem \ref{thm:main}.
Therefore $X\subset\p^M$ is a variety of minimal degree in its
linear span. 
\end{proof}

\begin{remark}\label{rem:p^N-k}
If $X\subset\g(k,N)$ is a non-degenerate embedding contained in a
Schubert variety $\Omega(k-1,N)$ or, equivalently,
$Q_X\cong\an_X^k\oplus\an_X(1)$, then
$X\subset\Omega(k-1,N)=\p^{N-k}\subset\p^M$ is
a non-degenerate embedding of $X$ in $\p^{N-k}$, and hence the bound
$\deg(X)\geq N-k+1-n$ cannot be improved.
\end{remark}

\begin{example}
Let us show some examples of subvarieties $X\subset\g(k,N)$ of
degree $\deg(X)=N-k$ and not contained in any $\Omega(k-1,N)$:
\begin{enumerate}
\item[(i)] Let $X=\p^n$. Consider the embedding $X=\g(n-1,n)$
given by $E=T_{\p^n}(-1)$.

\item[(ii)] Consider $X=\g(1,3)$, which is a hyperquadric in $\p^5$ by the Pl\"ucker embedding.

\item[(iii)] Let $X=\p^2$. Consider the embedding
$X\subset\g(1,5)$ given by $E=\an_{\p^2}(1)\oplus\an_{\p^2}(1)$. The
Pl\"ucker embedding of $X$ is a Veronese surface in $\p^5$.

\item[(iv)] Consider a Schubert variety $X=\Omega(1,N)\subset\g(1,N)$. The Pl\"ucker embedding
of $X$ is a cone of vertex a point over the Segre embedding
$\p^1\times\p^{N-2}$, so $\deg(X)=N-1$.
\end{enumerate}
\end{example}

We proceed with the proof of Corollary \ref{cor:N+1}:

\begin{proof}[Proof of Corollary \ref{cor:N+1}]
If $X\subset\g(k,N)$ is not contained in any $\Omega(s,N)$ then it
follows from Remark \ref{rem:omega} that $h^0(X,Q^*_X)\leq s$.
Therefore $$\deg(Q_X)=\deg(X)\leq N+1-s\leq
h^0(X,Q_X)-h^0(X,Q^*_X).$$ We deduce from Corollary \ref{cor:main}
that $p_a(C)\leq 1$ unless $C$ is hyperelliptic and
$Q_C\cong\an_C^r\oplus g^1_2\oplus\dots\oplus g^1_2$. But the
latter cannot occur since the morphism $\varphi:C\to\g(k,N)$
corresponding to $Q_C$ is not an embedding. Indeed,
$\varphi:C\to\varphi(C)\subset\p^M$ is a double covering of a
rational normal curve of degree $k+1-r$. 
\end{proof}

\begin{remark}\label{rem:p^k+1}
If $X\subset\g(k,k+1)=\p^{k+1}$ then Corollary \ref{cor:N+1} does
not hold, since the assumption $h^0(X,Q_X)>e+1$ in Corollary
\ref{cor:main} is not satisfied.
\end{remark}

\begin{example}
Some examples of Del Pezzo manifolds in $\g(k,N)$ of
degree $N+1$ and not contained in any $\Omega(0,N)$, thus on the boundary of
Corollary \ref{cor:N+1}, are given by the intersection of
$\g(1,3)\subset\p^5$ with a general hyperquadric of $\p^5$, or by $X=\g(1,4)$.
\end{example}


\begin{acknowledgement}
This research was begun while I was visiting the Mathematics
Department at Universit\`a degli Studi Roma Tre in the fall of
2006. I would like to thank Angelo F. Lopez for his warm
hospitality during the three months I spent there, as well as Mike
Roth for pointing out the simple proof of Proposition
\ref{prop:key}. Finally, I would especially like to thank the
referee for many helpful comments and suggestions.
\end{acknowledgement}

%

\end{document}